\documentclass[11pt]{article}
\usepackage{amssymb}
\usepackage{amsmath}
\usepackage{amsthm}
\usepackage[dvips]{graphicx}
\usepackage{enumitem}
\usepackage{color}

\usepackage{fullpage}

\theoremstyle{plain}
\newtheorem{thm}{Theorem}[section]
\newtheorem{coro}[thm]{Corollary}
\newtheorem{prop}[thm]{Proposition}
\newtheorem{lem}[thm]{Lemma}

\newtheorem{fact}[thm]{Fact}
\newtheorem{defi}[thm]{Definition}

\theoremstyle{definition}
\newtheorem{exam}[thm]{Example}
\newtheorem*{notation}{Notation}
\newtheorem{constr}[thm]{Construction}
\newtheorem{rem}[thm]{Remark}
\newtheorem{quest}[thm]{Question}

\newcommand{\R}{\mathbb{R}}
\newcommand{\Q}{\mathbb Q}
\newcommand{\N}{\mathbb N}
\newcommand{\Z}{\mathbb Z}

\DeclareMathOperator{\CO}{CO}
\DeclareMathOperator{\LO}{LO}
\DeclareMathOperator{\rot}{rot}
\DeclareMathOperator{\Homeo}{Homeo}
\DeclareMathOperator{\Hom}{Hom}
\DeclareMathOperator{\ord}{ord}
\DeclareMathOperator{\id}{\mathit{id}}

\newcommand{\vsp}{\vspace{0.1cm}}
\newcommand{\vs}{\vspace{0.3cm}}

\begin{document}

\title{On the  number of circular orders on a group}

\author{Adam Clay, Kathryn Mann and Crist\' obal Rivas \footnote{A. Clay was partially supported by NSERC grant RGPIN-05465, K. Mann was partially supported by NSF grant DMS-1606254 and C. Rivas was partially supported by FONDECYT 1150691. }  }

\maketitle

\begin{abstract}
We give a classification and complete algebraic description of groups allowing only finitely many (left multiplication invariant) circular orders.  In particular, they are all solvable groups with a specific semi-direct product decomposition.    This allows us to also show that the space of circular orders of any group is either finite or uncountable.  As a special case and first step, we show that the space of circular orderings of an infinite Abelian group has no isolated points, hence is homeomorphic to a cantor set.  

\end{abstract}


\section{Introduction}

Let $G$ be a group.  A \emph{left order} on $G$ is a total order invariant under left multiplication, i.e. such that $a<b \Leftrightarrow ga<gb$ holds for all $a,b,g \in G$.  It is well known that a countable group is left-orderable if and only if it embeds into the group $\Homeo_+(\R)$ of orientation-preserving homeomorphisms of $\R$. Similarly, a \emph{circular order} on a group $G$ is defined by a left-invariant \emph{cyclic orientation cocycle} $c: G^3 \to \{\pm1, 0\}$ satisfying a ``nondegenerate cocycle" condition (see definition \ref{cocycle def}), and for countable groups this is equivalent to the group embedding into $\Homeo_+(S^1)$ (see Proposition \ref{prop dyn real}).


Understanding the class of groups allowing only finitely many (left or circular) orders is a natural question both from the group-order perspective and also from the dynamical perspective as these groups present strong constrains on their action on 1-manifolds. In fact, the sets $\LO(G)$ and $\CO(G)$ of all left or circular orders on a group carry a natural Hausdorff topology (see Section \ref{sec preliminaries}) and, as proved by the second and third author in \cite{MR}, isolated points therein corresponds to actions having a strong form of rigidity.


In \cite{tararin}, Tararin gave a complete classification of all groups $G$ with only finitely many left orders.  These groups are now known as \emph{Tararin groups}, and are described in Section \ref{sec Tararins} below.   Linnell \cite{linnell} later used Tararin's work, along with the dynamics of the conjugation action of $G$ on $\LO(G)$, to show that, for any group $G$, the set $\LO(G)$ is either finite or uncountable.  Navas \cite{navas} later gave a different proof of Linnell's result for countable groups using the dynamics of actions of $G$ on $\R$ instead of on $\LO(G)$  (an insight later generalized for any group in \cite{NRC}). In this approach Tararin's classification was also crucial. 
However, neither Linnell's nor Navas' proof sheds much light on spaces of circular orders.   In fact, for a fixed group $G$, the spaces $\CO(G)$ and $\LO(G)$ are typically very different.  As a concrete example, we note in Section  \ref{sec Tararins} that every Tararin group admits \emph{infinitely} many circular orders.  

The purpose of this work is to give a complete classification of groups admitting only finitely many circular orders, and a proof that for any group $G$, the space $\CO(G)$ is finite or uncountable.  Precisely, we show: 

\vs
\noindent{\bf Theorem A.} 
{\em Let $G$ be a circularly orderable group. Then $\CO(G)$ is either finite or uncountable. Moreover, the following two conditions are equivalent:
\begin{enumerate}
\item The group $G$ admits only finitely many circular orderings.
\item Either $G$ is finite cyclic, or $G\simeq T\rtimes \Z_n$, where $n \in 2\Z$, $T$ is a Tararin group with maximal convex subgroup $T_1$, and $\Z_n$ acts by inversion on $T/T_1$.
\end{enumerate}
}

\noindent Here $\Z_n = \Z/n\Z$ denotes the cyclic group of order $n$.  The definition of \emph{maximal convex subgroup} is reviewed in Section \ref{sec preliminaries}. Note that,  since Tararin groups are solvable, Theorem A implies in particular that a group with only finitely many circular orders is solvable.

A major tool in the proof is the notion of the \emph{linear part} of a circular order (see Proposition \ref{prop linear part}), which was introduced in \cite{MR} (although a similar definition in the context of bi-invariant circular orders on groups appears in \cite{JP}).   Another key tool is the result of D. Witte-Morris \cite{morris} that an amenable left-orderable group must be {\em locally indicable}, that is, each of its finitely generated subgroups admits a surjection to $\Z$.  We make extensive use of this in order to show that the linear part of any circular order on a group of the form $G\simeq T\rtimes \Z_n$ (as in Theorem A) is in fact a Tararin group. See for instance Proposition \ref{prop tararin char}.

As an intermediate step in the proof, we study orders on Abelian groups and prove 

\vs
\noindent{\bf Theorem B.} {\em  The space of circular orders on an infinite Abelian group $A$ has no isolated points. In particular, when $A$ is countably infinite $\CO(A)$ is homeomorphic to a Cantor set.}
\vs

This result was proved for the special case of finitely generated Abelian groups (using, in particular, that these have finite torsion subgroup) in \cite[\S 3]{BS}.  Their methods are somewhat different and involve a partition of the space of $\CO(G)$ using the dynamics of induced actions on $S^1$.   
\vs

Theorem A raises the obvious problem of computing $|\CO(G)|$ in the cases where it is finite.  

\begin{quest}  Given $G = T \rtimes \Z_n$ as in Theorem A, can $|\CO(G)|$ be expressed in terms of the algebraic structure of $G$? 
\end{quest}

It is easy to see that $|\LO(T)|$ is a lower bound, and to give examples of groups where $|\LO(T)| = |\CO(G)|$, and where $|\LO(T)| < |\CO(G)|$ hold.  We expect that the tools developed here (e.g. the notion of ``inverting base" for $G$) provide a good framework for answering this question; see Section \ref{number sec} for further discussion.

\paragraph{Outline.} 
The organization of this paper is as follows. In Section \ref{sec preliminaries} we review definitions and known results, and prove a variation of H\"older's theorem that holds for circularly ordered groups.  Section \ref{sec abelian} contains the proof of Theorem B which is latter used in Section \ref{cardinality sec}. In Section \ref{sec Tararins}   we show that assertion 2 implies assertion 1 in Theorem A, this is the most technical part of the work.   Finally, in Section \ref{cardinality sec} we show that 1 implies 2 in Theorem A, and simultaneously show that there are no groups whose space of orders is countably infinite.

\section{Preliminaries}
\label{sec preliminaries}

\begin{defi}  \label{cocycle def}
Let $S$ be a set. A \emph{circular order} on $S$ is a function
$ c: S^3\to\{\pm1,0\}$ satsifying
 \begin{enumerate}[label=\roman*)]
\item $c^{-1}(0)=\triangle(S)$, where $\triangle(S):=\{(a_1,a_2,a_3)\in S^3\mid a_i=a_j,\text{ for some } i\neq j\}$,
\item $c$ is a cocycle, that is $c(a_2,a_3,a_4)-c(a_1,a_3,a_4)+c(a_1,a_2,a_4)-c(a_1,a_2,a_3)=0$ for all $a_1,a_2,a_3,a_4\in S$.

\end{enumerate}
A group $G$ is {\em circularly orderable} if it admits a circular order $c$ which is left-invariant in the sense that $c(a_1,a_2,a_3) = c(ga_1, ga_2, ga_3)=1$ for all $g,a_1,a_2,a_3 \in G$.   
\end{defi}

The topology on $\CO(G)$ is the subset topology from the function space $\{1, -1, 0\}^{G \times G \times G}$.  A neighborhood basis of $c \in \CO(G)$ is given by sets of the form 
$$O_S := \{ c' \in \CO(G) : c(s_1,s_2,s_3) = c'(s_1,s_2,s_3) \text{ for all } s_1,s_2,s_3 \in S\}$$ 
as $S$ ranges over all finite subsets of $G$. This makes $\CO(G)$ a totally disconnected, compact space, which is metrizable whenever $G$ is countable. By convention we say that two circular orders $c$ $c'$ on $G$ \emph{agree on} $S\subseteq G$, if $c_1(s_1,s_2,s_3)=c_2(s_1,s_2,s_3)$ holds for all $s_1,s_2,s_3\in S$.

In the same spirit, a neighborhood basis of $<$ in $\LO(G)$ is given by sets of the form $\{ <' : s_1 <' s_2 \Leftrightarrow s_1< s_2 \text{ for all } s_1, s_2 \in S\}$, where $S$ again ranges over finite sets.

The remainder of this section discusses various ways to build left-orders from circular orders, and vice versa.  
We start with some definitions.

\begin{defi}\label{def order by restriction}
Let $(G,c)$ be a circularly ordered group.  A subgroup $H \subset G$ is said to be \emph{left-ordered by restriction} if we have $c(h^{-1}g^{-1}, id,  gh) = 1$ whenever $c(h^{-1}, id, h) = c(g^{-1}, id, g) =1$.  When $H \subset G$ is left-ordered by restriction, we define the \textit{restriction ordering} of $H$ to be the left-order on $G$ with positive cone (i.e. the set of elements greater than identity) given by
\[ P = \{ h \in H \mid c(h^{-1}, id, h) =1 \}.
\]
\end{defi}

\begin{defi}
A subgroup $H$ of a circularly ordered group $(G,c)$ is \emph{convex} if $H$ is left-ordered by restriction and whenever one has $c(h_1, g, h_2) = 1$ and $c(h_1, id, h_2) =1$ with $h_1, h_2$ in $H$ and $g \in G$, then $g \in H$.  

This generalizes the standard definition of convex for left ordered groups; recall that a subgroup $H$ in $(G, <)$ is convex if for any $h_1,\, h_2\in H$ and $g\in G$, if $h_1 < g < h_2$ then $g \in H$.  
\end{defi}

\begin{prop}\label{prop extensions}
If $(G, c)$ is a circularly ordered group and $H$ is convex:
\begin{enumerate} 
\item The cosets $G/H$ inherit a $G$-invariant circular ordering $c_H$ defined by $c_H(g_1H, g_2H, g_3H) = c(g_1, g_2, g_3)$ for every triple $(g_1H, g_2H, g_3H) \in (G/H)^3 \setminus \Delta(G/H)$.
\item If in addition $H$ is normal in $G$, then there is an injective continuous map 
$$\psi:\LO(H)\times \CO(G/H)\to\CO(G),$$ having $c$ in its image.
\end{enumerate}
\end{prop}
\begin{proof} We first prove 1. It suffices to check that $c_H$ is well-defined, as left-invariance follows from left-invariance of $c$.  

Note that since $c(g_1, g_2, g_3)$ is invariant under cyclic permutation of the $g_i$'s, it suffices to show that $c(g_1, g_2h, g_3) = c(g_1, g_2, g_3)$ for all $h \in H$ and all triples $(g_1, g_2, g_3) \in G^3$ such that $g_i H \neq g_j H$ for $i \neq j$.  For if this holds, we have
\begin{align*}
c(g_1h_1, g_2 h_2, g_3h_3)  &=  c(g_1h_1, g_2 , g_3h_3)= c(g_3h_3, g_1h_1, g_2)\\
&= c(g_3h_3, g_1, g_2) = c(g_2, g_3h_3, g_1)\\
& = c(g_2, g_3, g_1) = c(g_1, g_2, g_3).
\end{align*}
for any $h_1, h_2, h_3 \in H$.  

There are a few cases to check; we give the details of one.  
Let $h \in H$ and suppose that $c(h^{-1}, id, h) =1$.   Suppose for contradiction that $c(g_1, g_2, g_3) = c(g_2^{-1}g_1, \id, g_2^{-1}g_3) = 1$, but  $c(g_1, g_2h, g_3)=-1$ (Note that our assumptions imply $c(g_1, g_2h, g_3) \neq 0$.)   Then $c(g_3, g_2h, g_1) = 1$ and so $c(g_2^{-1}g_3, h , g_2^{-1}g_1) = 1$.  Combining this with $c(g_2^{-1}g_1, \id, g_2^{-1}g_3) = 1$ we get $c(1, g_2^{-1}g_3, h) = 1$.  This implies $c(h^{-1}, g_2^{-1}g_3, h) = 1$, so that $g_2^{-1} g_3 \in H$ by convexity.  This contradicts our assumption that $g_2H \neq g_3H$.  The case $c(h^{-1}, id, h) =-1$ is similar, as is the case $c(g_1, g_2h, g_3)=1$ and $c(g_1, g_2, g_3) =-1$.  Thus $c_H$ is well-defined.

We now prove 2. Suppose that $<_H$ is a left order on $H$, and $\bar c $ a circular order on $G/H$.  Define $c'(g_1,g_2,g_3)=1$ whenever $\bar{c}(g_1H,g_2H,g_3H)=1$. When $(g_1,g_2,g_3)$ is a non-degenerate triple but $g_1H=g_2H\neq g_3H$, then one can declare $c(g_1,g_2,g_3)=1 $ whenever $g_2^{-1}g_1 <_H id$.  This determines also the other cases where exactly two of the elements $g_iH$ coincide.   The remaining case is when $g_1H=g_2H=g_3H$, in which case we declare $c(g_1,g_2,g_3)$ to be the sign of the permutation $\sigma$ of $\{ \id, g_1^{-1}g_2, g_1^{-1}g_3\}$ such that $\sigma(\id) < \sigma(g_1^{-1}g_2) < \sigma(g_1^{-1}g_3)$.   Checking that this gives a well-defined left-invariant circular order is easy and left to the reader (see details in \cite[Theorem 2.2.14]{Calegari}).  Certainly, if $\bar c =c_H$ and $\leq_H$ is the is restriction of $c$ to $H$ (see Definition  \ref{def order by restriction}), then $c'$ coincides with $c$. So $\psi$ is an injection having $c$ in its image.

Continuity of $\psi$ is easy. If $S$ is a finite subset of $G$ on which we want to approximate $c'=\psi(<_H,\bar{c})$, then it is enough to pick $\bar{\bar c}$ agreeing with $\bar c$ over $SH=\{sH\mid s\in S\}$, and $\widetilde{<}_H$ agreeing with $<_H$ over $\{s_1s_2\mid s_1,\,s_2\in S\cup S^{-1}\}$.\end{proof}

We will frequently use the following result from \cite{MR}.

\begin{prop}[\cite{MR}] Every circularly ordered group $(G,c)$ admits a unique maximal convex subgroup $H\subseteq G$. 
\label{prop linear part}
\end{prop}

Following \cite{MR}, we call $H$ the \emph{linear part} of $c$.   We now present two constructions from \cite{zheleva}.

\begin{constr} \label{construction quotient}  Let $(G, <)$ be a left-ordered group with a central element $z > \id$.  Assume $z$ is {\em cofinal} (i.e. for all $g\in G$ there is $n\in \Z$ such that $z^{-n}<g<z^n$).  We can define a circular order on $G/\langle z\rangle$ as follows:  for every $\bar g\in G/\langle z\rangle$, define the {\em minimal representative} of $\bar{g}$ to be the unique $g\in \bar{g}$ satisfying that $id\leq g<z$. We then define 
$$c(\bar{g}_1, \bar{g}_2,\bar{g}_3)=sign(\sigma),$$
where $\sigma$ is the unique permutation of $S_3$ such that $g_{\sigma(1)}<g_{\sigma(2)}<g_{\sigma(3)}$, where $g_i$ ($i=1,2,3$) is the minimal representatives of $\bar g_i$.
\end{constr}

\begin{constr} \label{construction lift}  Conversely, given a circular ordered group $(G,c)$, there is a left ordered group $(\widetilde{G},<)$ containing a positive, central and cofinal element $z$, such that $c$ is obtained from $<$ through Construction \ref{construction quotient} above. Indeed, $\tilde{G}$ can be taken to be the central extension of $G$ by $\Z$ given as the set $G\times \Z$, endowed with the multiplication
$(a,n)(b,m)=(ab,n+m+f_{a,b})$, where
$$ f_{a,b}=\left\{\begin{array}{cl} 0 & \text{if $a=\id$ or $b=\id$ or $c( \id, a, ab)=1$}
\\ 1 &  \text{if $ab=id$ $(a\not=id)$ or $c(id,ab,a)=1$.  }  \end{array}
\right.$$ 
Clearly, the element $z:=(\id,1)$ is central, and if we define the set of non-negative elements for $\leq$ as $P=\{(a,n)\mid n\geq 0\}$, then $z$ is also cofinal.  It is easy to check that applying Construction \ref{construction quotient} recovers $c$.  
\end{constr}

\begin{rem}    \label{functorial rem}
Construction \ref{construction lift} respects morphisms in the following sense:  if 
$(G_i,c_i)$, $i=1,2$, are circularly ordered groups, and $\phi:G_1\to G_2$ is a homomorphism such that $c_1(g_1, g_2, g_3) = c_2(\phi(g_1), \phi(g_2), \phi(g_3))$ for all $(g_1, g_2, g_3) \in G_1^3$, then there is an order preserving homomorphism $\tilde{\phi}:(\tilde G_1, <_1)\to (\tilde G_2,<_2)$ which is injective (or surjective) if $\phi$ is so.  Indeed, we can define $\tilde{\phi}: \widetilde{G}_1 \rightarrow \widetilde{G}_2$ by $\tilde{\phi}(g,n) = (\phi(g), n)$ for all $(g,n) \in G_1 \times \Z$.  That $\tilde{\phi}$ is a homomorphism follows from the observation that $f_{\phi(g_1), \phi(g_2)} = f_{g_1, g_2}$ for all $g_1, g_2 \in G_1$, since $\phi$ respects the circular orderings $c_1, c_2$.  It is also clear that $\tilde{\phi}$ respects the orderings $<_1$ and $<_2$ defined above, as it maps the positive cone of $<_1$ into the positive cone of $<_2$.
\end{rem}

\subsection{Archimidean circular orders}

Archimidean orders will play an important role in our discussion of orders on Abelian groups.  We recall the definitions. 

\begin{defi}[Stolz \cite{Stolz}]
A left-ordered group $(G, <)$ is Archimedean if, for all $f, g \in G$, such that $f \neq \id$, there exists $n \in \Z$ with $f^n > g$.
\end{defi}  

\begin{defi}[Swierczkowski \cite{swier}] A circularly ordered group $(G,c)$ is called {\em Archimedean} if there are no elements $f$, $g$ in $G$ such that $c(id,g^n,f)=1$ for all $n \geq 1$.
\end{defi}

\noindent A different definition of Archimedean circular orderings appears in \cite{BS}.  However, \'Swierczkowski's appears more natural, as the definition in \cite{BS} excludes circular orders on $\Z$.  

\begin{prop}
\label{lifting prop}
Let $(G,c)$ be a circularly ordered group and let $(\widetilde{G}, <)$ denote the left-ordered group given by Construction \ref{construction lift}. Then:
\begin{enumerate}[nolistsep]
\item $(\widetilde{G}, <)$ is an Archimedean left-ordered group if and only if $(G,c)$ is an Archimedean circularly ordered group.
\item $(\widetilde{G}, <)$ is a bi-ordered group if and only if $(G,c)$ a bi-invariant circularly ordered group.
\end{enumerate}
\end{prop}
\begin{proof} We first prove that if $(\widetilde{G}, <)$ is not Archimedean, then neither is $(G,c)$.  
Observe that if $\id < \tilde g \in \widetilde{G}$ and there is $n\in \N$ with $z<\tilde{g}^n$, then $\tilde{g}$ is cofinal. Therefore, if $(\tilde G,<)$ is not Archimedean, then we can find $\tilde{g}$ and $\tilde{k}$ such that $id<\tilde{g}^n<\tilde{k} < z$ for all $n\in \N$.  In particular, the relationship between Constructions  \ref{construction quotient} and \ref{construction lift} implies that
$c(\langle z\rangle, \tilde{g}^n\langle z\rangle, \tilde{k}\langle z\rangle)=1, \forall n\in \N,$
hence $(G,c)$ is not Archimedean.

Conversely, if $(G,c)$ is not Archimedean, there exist elements $g, h \in G$ such that $c(id, g^n, h) =1$ for all $n \geq 1$.  Then, using the notation from Construction \ref{construction lift}, we have $(g^n,0)^{-1}(h,0) = (g^{-n}h, -1 +0 +f_{g,h})$.  Since $c(id, g^{-n}h, g^{-n}) =1$ for all $n \geq 1$, we have $f_{g,h} =1$ for all $n \geq 1$.  Thus $((g^n,0)^{-1}(h,0)  = (g^{-n}h, 0) > \id$ for all $n$, so $(\widetilde{G}, <)$ is not Archimedean.  This proves (1).

Next, suppose that $(\widetilde{G}, <)$ is not bi-orderable.   Choose elements $\tilde{g}, \tilde{h} \in \widetilde{G}$ with $id<\tilde{g}$ and $\tilde{h}^{-1}\tilde{g}\tilde{h}<id$.  We may assume (upon multiplying by powers of $z$ if necessary) that $id<\tilde{h}<z$.  Note we must also have $\tilde{g}<\tilde{h}$, as $\tilde{h}<\tilde{g}$ implies $\tilde{h}^{-1}\tilde{g}\tilde{h}>id$.   From this it follows that 
$\tilde{g}^2\tilde{h}<\tilde{g}\tilde{h}<\tilde{h}$.  Combining this with $\tilde{g}^2<\tilde{g}^2\tilde{h}$, we have
\[ id<\tilde{g}<\tilde{g}^2<\tilde{g}^2\tilde{h}<\tilde{g}\tilde{h}<\tilde{h}<z,
\]
so all the elements in the inequalities above are minimal representatives.  Thus the circular ordering $c$ satisfies $c(id, g, g^2) =1 $ while $c(h, gh, g^2h) =-1$, so it is not right-invariant.

Conversely, suppose $(\widetilde{G}, <)$ is bi-orderable and that $c(f, g, h) =1$ for some elements $f, g, h \in G$ with minimal coset representatives $\tilde{f}, \tilde{g}, \tilde{h} \in \widetilde{G}$.  Then either 
\[\tilde{f}< \tilde{g}< \tilde{h} \mbox{ or } \tilde{h}< \tilde{f}< \tilde{g} \mbox{ or } \tilde{g} <\tilde{h}< \tilde{f}.
\]
Consider the case $\tilde{f}< \tilde{g}< \tilde{h}$, the others are similar.  Let $a \in G$ with minimal coset representative $\tilde{a}$.  Then $id< \tilde{f}\tilde{a}< \tilde{g}\tilde{a}< \tilde{h}\tilde{a} < z^2$.  
Suppose that $\tilde{f}\tilde{a}< \tilde{g}\tilde{a}<z< \tilde{h}\tilde{a}$, the cases $\tilde{f}\tilde{a}<z< \tilde{g}\tilde{a}< \tilde{h}\tilde{a}$ and $z< \tilde{f}\tilde{a}< \tilde{g}\tilde{a}< \tilde{h}\tilde{a}$ are similar.  Then the minimal coset representatives satisfy one of
 \[\tilde{f}\tilde{a}< \tilde{g}\tilde{a}< \tilde{h}\tilde{a}z^{-1} \mbox{ or } \tilde{f}\tilde{a}< \tilde{h}\tilde{a}z^{-1} < \tilde{g}\tilde{a} \mbox{ or } \tilde{h}\tilde{a}z^{-1} < \tilde{f}\tilde{a}< \tilde{g}\tilde{a}.\]
 The first two inequalities above are impossible as right-multiplying by $\tilde{a}^{-1}$ yields $\tilde{f}<\tilde{h}z^{-1}$ in each case, and $\tilde{h}z^{-1} < \id$.  Thus it must be that $\tilde{h}\tilde{a}z^{-1} < \tilde{f}\tilde{a}< \tilde{g}\tilde{a}$, and so $c(ha, fa, ga) = c(fa, ga, ha)=1$.
\end{proof}

We immediately obtain the following corollary, a special case of which (for bi-invariant circular orders) was proved in \cite{swier}. 

\begin{coro} \label{coro holder} Every Archimedean circularly ordered group is (order) isomorphic to a subgroup of $S^1$.
\end{coro}
For the proof recall that the natural circular order \emph{ord} on $S^1$ is obtained by declaring 
$$\ord(x,y,z) = \left\{\begin{array}{rl} -1 & \text{if } (x, y, z) \text{ is clockwise oriented} \\ 1 & \text{if } (x, y, z) \text{ is counter clockwise oriented} 
\\ 0 & \text{if any two of } x, y \text{ and } z \text{ agree}
.\end{array}
\right. $$ 
Equivalently, this is obtained by applying Construction \ref{construction quotient} to $\R$ with the standard left order $<$ to get a circular order on $\R/\Z = S^1$.   

\begin{proof}
If $G$ is Archimedean circularly ordered, then $\widetilde{G}$ is equipped with an Archimedean ordering by Proposition \ref{lifting prop}(1).  By H\"older's theorem (\cite{Holder}, see \cite{GOD} for a modern exposition) 
 it therefore is (order) isomorphic to a subgroup of $(\R, +)$.  Moreover, after applying an automorphism of $\R$, we may normalize the embedding $\widetilde{G} \subset \R$ so that the central, cofinal element $z$ is mapped to the generator $1$ of  $\Z$.  Then $G = \widetilde{G}/ \langle z \rangle \subset \R/\Z$, and so is (order) isomorphic to a subgroup of $S^1$.
\end{proof}

The corollary below is a variation of a result of Jakub\'{i}k and P\'{e}stov \cite{JP}.  We will use it in Section \ref{cardinality sec}.  

\begin{coro}
\label{finite index}
Let $(G,c)$ be a circularly ordered group, and $H$ its linear part.  Suppose that there is a finite-index subgroup $K$ of $G$ such that $c$ is invariant under right multiplication by elements of $K$.  Then $H$ is normal, and $G/H$ is Archimedean circularly ordered. 
\end{coro}

The proof uses some standard facts about \emph{Conradian left-orders} \footnote{A Conradian order is a left order that satisfies for all $f>id$, and all $g>id$ there is $n\in \N$ such that $fg^n>g$.}.   Since these do not appear elsewhere in this work, we do not give a detailed explanation, referring the reader to \cite{NRC} for details.  

\begin{proof}
The inclusion map $K \rightarrow G$ of circularly ordered groups induces, by Remark \ref{functorial rem} an inclusion $\widetilde{K} \rightarrow \widetilde{G}$ of the left-ordered groups obtained by Construction \ref{construction lift}. 
By Proposition \ref{lifting prop}(2), $\widetilde{K}$ is bi-ordered by the restriction of the ordering of $(\widetilde{G}, <)$.  Since $[G:K] < \infty$ we have  $[\widetilde{G}:\widetilde{K}] < \infty$.  It follows from \cite{rolfsen,NRC} that $(\widetilde{G}, <)$ is a Conradian ordered group. 

Let $C \subset \widetilde{G}$ be the largest (i.e. the union of all) convex subgroup of $\widetilde{G}$ that does not contain its cofinal central element $z$. Then there are no convex subgroups between $C$ and $\widetilde{G}$, and so, since $\leq$ is Conradian, the subgroup $C$  is normal in $\widetilde G$  and $\widetilde{G}/C$ is an Archimedean ordered group (see for instance \cite{conrad}\cite[\S 3.2]{GOD}). 

Since $C$ maps to $H$ under the quotient $\widetilde{G} \rightarrow G$, the result follows.
\end{proof}

\section{Ordering Abelian groups}  \label{sec abelian}

Using the results of the previous section, we describe the space of orders on Abelian groups.  It is well-known, and easy to show, that a finite Abelian group $A$ admits a circular order if and only if $A$ is cyclic (see for instance \cite{Ghys}).  In this case, $\CO(A)$ has cardinality $\phi(|A|)$, where $\phi$ is Euler's totient function.    For infinite groups, we will show

\begin{thm} \label{teo abelian}The space of circular orders of an infinite Abelian group has no isolated points.  
\end{thm}

In particular, if $A$ is countably infinite, then $\CO(A)$ is a compact, totally disconnected, metrizable, perfect space; hence homeomorphic to the Cantor set.   This gives Theorem B from the introduction.

The first step in the proof of Theorem \ref{teo abelian} is the following analog of  H\"older's theorem for left-ordered groups.

\begin{lem}[\textit{cf.} \'Swierczkowski \cite{swier}] \label{lem holder}Let $(A,c)$ be an Abelian circularly ordered group and let $H$ be its linear part. Let $\bar{c}$ be the induced circular order on $A/H$. Then, $(A/H,\bar c)$ is order isomorphic to a subgroup of $S^1$.

\end{lem}

\noindent{\bf Proof:} By Corollary \ref{coro holder}, it is enough to check that $(A/H,\bar{c})$ is Archimedean. Suppose it is not. Then there are $\bar g,\, \bar h \in A/H$ such that $\bar c (\bar{id}, \bar{h}^n,\bar g)=1$ for all $n\in \N$. Let 
$$\bar I=\{\bar f \in G/H\mid \text{ there is $n\in \N$ with } \bar{c}(\bar{h}^{-n}, \bar{f},\bar{h}^n)=1\}.$$
Since $A$ is Abelian, $\bar{I}$ is a convex subgroup which is proper by the condition on $\bar h$. This implies that $\bar{I}=I/H$ for some convex subgroup $I$. This contradicts the fact that $H$ was the maximal convex subgroup of $A$. $\hfill\square$

\begin{lem}  \label{S1 nonisolated}
Let $G \subset \R/\Z = S^1$ be an infinite group, with the induced circular order from $S^1$.   Then this order is not isolated in $\CO(G)$.
\end{lem} 

\begin{proof}  We have two cases. 

\noindent \textit{Case i. $G$ is not a torsion group.}
Let $\hat{G} \subset \R$ be the set of lifts of elements to $\R$, so we have a short exact sequence $0 \to \Z \to \hat{G} \to G \to 0$, and consider the vector space $V$ over $\Q$ generated by $\hat{G} \subset \R$. By assumption, $G$ is not torsion, so $V \neq \Q$.  Let $\lambda \in V \setminus \Q$.  If $V \neq \R$, we can choose $\lambda'$ linearly independent over $\Q$ from $V$, and define $\phi: V \to \R$ by $\lambda q \mapsto \lambda' q$ for $q \in \Q$, and $\alpha \mapsto \alpha$ for any $\alpha$ in the complement of the span of $\lambda$.    If $V = \R$, instead choose two basis elements $\beta$ and $\gamma$ and define $\phi: \R \to \R$ by $\beta q \mapsto \gamma q$, $\gamma q \mapsto \beta q$, and $\alpha \mapsto \alpha$ for any $\alpha$ in the complement of the span of $\{ \beta, \gamma \}$.  

In either case, $\phi$ is an isomorphism of $\hat{G}$ and descends to an embedding of $G$ in $\R/\Z$ with a different cyclic order.  This order can be made arbitrarily close to the original one by taking $\lambda'$ as close as we like to $\lambda$, or by choosing elements $\beta$ and $\gamma$ in a basis for $\R$ over $\Q$ that are arbitrarily close together.  
\medskip

\noindent \textit{Case ii. $G \subset \Q/\Z$. }
In this case $G$ is an infinite torsion group and we can decompose it into a direct sum of groups $G = \oplus G_p$, where $G_p$ is the group of all elements with order a power of $p$, for each prime $p$, i.e. the elements $a/p^k \in \R/\Z$.  We use the following basic fact. 

\begin{fact}
Let $A_p = \{a/p^k : k \in \N \} \subset \R/\Z$.   Then, for any $k$, the function $x \mapsto x+ p^k x$ is an automorphism of $A_p$.  
\end{fact} 
To see this, one checks easily that any map
$x \mapsto \sum_{i=0}^\infty a_i p^i x$ 
gives a well defined endomorphism:  for fixed $x \in A_p$, all but finitely many terms in the formal power series $\sum a_i p^i x$ vanish mod $\Z$.  Since $1 + p^k$ is invertible in the $p$-adic integers, $x \mapsto x+ p^k x$ has an inverse, so is a homomorphism.  

Now given any finite subset $S$ of $G$, we can find some $G_p \subset G$ and $N > 0$ such that $G_p \cap \{a/p^k : k>N \}$ is nonempty and does not contain any element of $S$.   Let $k$ be the smallest integer greater than $N$ such that $G_p$ contains an element of the form $a/p^k$.   Using the fact above, define a homomorphism $G \to \R/\Z$ to be the identity on $G_q$ for $q \neq p$, and to be $x \mapsto x+ p^{k-1} x$ on $G_p$.    Note that this is well defined, injective, restricts to the identity on $S$, and changes the cyclic order of elements of $G_p$. 
\end{proof} 

\noindent{\bf Proof of Theorem \ref{teo abelian}:}  
Fix a circular order $c$ on an infinite Abelian group $A$, and let $H$ be its linear part. We will show that $c$ is not isolated.  
Since $H$ is convex, Proposition \ref{prop extensions} gives a continuous embedding $\LO(H)\times \CO(A/H)\hookrightarrow \CO(A)$ with $c$ in its image.  
By Lemma \ref{lem holder} the induced order on $\CO(A/H)$ is order isomorphic to a subgroup of $S^1$, so the Lemma \ref{S1 nonisolated} shows that it is either non-isolated (in which case we are done), or $A/H$ is finite.   Also, if $H$ has rank\footnote{Recall that the rank of an torsion free Abelian group $A$ is the least $n$ such that $A$ is isomorphic to a subgroup of $\Q^n$. For general Abelian group, the rank is defined as the rank of its quotient by its torsion subgroup.} greater than one then $\LO(H)$ has no isolated points (see \cite{sikora}), so $c$ is not isolated in that case either.  Thus, we assume that $A/H$ is finite, say $|A/H | = k$, and $H$ has rank one, i.e. $H\subset \Q$. We may assume that the restriction of $c$ to $H$ is just the induced order from $\Q$.

Let $S\subset H$ be a finite set. Pick $M \in \R \setminus \Q$ such that $S \subset  (-M, M)$ and define an injective homomorphism $\phi: \Q \to S^1 = \R/\Z$ by $\phi(r) = r/2kM$.  Note that cyclic order of $H$ in this new embedding into $S^1$ agrees with $c$ on the subset $S$.  

Now we extend this to a map $A \to \R/\Z$ as follows.  Let $t$ be a generator for $|A/H| \cong \Z/k\Z$ such that $id, t, t^2, t^3, ..., t^{k-1}$ are in positive cyclic order.   Let $\hat{t} \in A$ be any element that projects to $t \in A/H$.   Then each element in $A$ can be written uniquely as $h \hat{t}^n$ for some $h \in H$, and $0 \leq n \leq k-1$   Extending our previous map of $H$ into $S^1 = \R/\Z$, define a map $\Phi: A \to S^1$ by 
$$h \hat{t}^n \mapsto \phi(h) + n/k.$$
We claim that this is an injective homomorphism.  To see this, we compute
$$\Phi(h_1 \hat{t}^n h_2 \hat{t}^m) = \Phi(h_1h_2 \hat{t}^{n+m}) = \phi(h_1)\phi(h_2) \tfrac{n+m}{k} = \Phi(h_1 \hat{t}^n) \Phi(h_2 \hat{t}^m).$$   Injectivity follows from 
$$\Phi(h \hat{t}^n) = 0 \Rightarrow h/2kM + n/k = 0 \text{ mod } 1$$
(here we think of $h \in \Q$).  Since $h/2kM \notin \Q$, this implies that both $h = 0$ and $n = 0$.  

Now it is easy to verify that $\Phi$ gives a cyclic order that agrees with $c$ on the union of all sets of the form $S \hat{t}^n$, for $0 \leq n \leq k-1$.  Within one set $S \hat{t}^n$,  this follows since $\Phi(s \hat{t}^n) = s/2kM + n/k$, we have $s \in (-M, M)$, and $\Phi$ preserves order on $S$.  Moreover,  for any $s_i \in S$, under $c$ we have $s_0, s_1\hat{t}, ..., s_k \hat{t}^k$ in positive cyclic order under $c$, and by construction the same holds in the image of $\Phi$.  
\hfill \qed

\section{Classification of groups with finitely many circular orderings}

\label{sec Tararins}

In this section we give a sufficient condition for a group to admit only finitely many circular orderings.  In Section \ref{cardinality sec} we will show this condition is also necessary, thus giving a complete classification of groups admitting finitely many circular orders.  
We will frequently be discussing rank one, torsion-free Abelian groups and their automorphisms, and so for ease of discussion we often will implicitly identify such a group with a subgroup of $\mathbb{Q}$, and an automorphism of such a group with multiplication by some rational number.

Our classification is analogous to Tararin's classification of groups allowing only finitely many left orders (cf. \cite{tararin} or \cite[Theorem 5.2.1]{kopitov medvedev}), which we recall here for future use.

\paragraph{Tararin groups.} 
Recall that a series
$\{1\}=G_m\lhd G_{m-1} \lhd \ldots \lhd G_0=G$ is said to be
{\em rational} if each quotient $G_{i}/G_{i+1}$ is a
rank-one, torsion-free Abelian group. Observe that then $G$ is a solvable group of finite rank\footnote{A solvable group $\Gamma$ has finite rank if all successive quotients $\Gamma_i/\Gamma_{i+1}$ in its derived series have finite rank as
Abelian groups. The rank of  $\Gamma$ is then defined as $\sum rank(\Gamma_i/\Gamma_{i+1})$. See \cite{robinson}.}.

\begin{thm}[Tararin \cite{tararin}] \label{teo tararin}{ Let $T$ be a left orderable
group. Then the following conditions are equivalent:
\begin{enumerate}
\item The group $T$ admits only finitely many left orderings. 
\item The group $T$ admits a unique (hence characteristic) rational series
$$ \{1\}=T_{m+1}\lhd T_{m} \lhd \ldots \lhd T_0=T$$
such that for every
$0\leq i\leq m-2$, there is an element of $T_i/T_{i+1}$ whose action
by conjugation on $T_{i+1}/T_{i+2}$ is by multiplication by a
negative rational number.
\end{enumerate}
}
\end{thm}
We call a group satisfying either of the above conditions a {\em Tararin group}.

\begin{rem} \label{rem orders on Tararins}
The possible left-orderings of a Tararin group $T$ as in (\ref{teo tararin}) are described as follows.   Since each quotient $T_i/T_{i+1}$ is rank-one Abelian, it admits exactly two distinct left-orders (arising from the restriction of the natural ordering of $\Q$, and the opposite or ``flipped" order).   For every $i=1,\ldots,m$, choose an ordering  $\preceq_i$  on $T_{i-1}/T_{i}$. Then we can
produce a left-ordering on $T$ by declaring $g \succ id$ if and only if $g \succ_i id$, where $i$ is the unique index such that $g \in T_{i-1} \setminus T_i$.

Tararin's theorem shows that these are the only possible left-orderings of $T$; in particular $|\LO(T)| = 2^{m+1}$.  Observe that in any of these orderings, the groups $T_i$ are
convex, and conversely, every convex subgroup is of this form. 
\end{rem}

Since Tararin groups all admit an infinite Abelian quotient with left-orderable kernel, Proposition \ref{prop extensions} and Theorem \ref{teo abelian} imply that every Tararin group admits infinitely many circular orders. 
However, we will show that certain finite extensions of Tararin groups have few left orders. The remainder of this Section is devoted to prove the following result, which gives one direction (the more difficult implication) in Theorem A as stated in the introduction.  

\begin{thm} \label{teo tararin circ}
Suppose that $G$ is a group of the form $G = T\rtimes \Z_n$, where $T$ is a Tararin group with maximal convex subgroup $T_1$ and $\Z_n$ acts by inversion on $T/T_1$.  Then $G$ admits only finitely many circular orders.  
\end{thm}

\subsection{Algebraic structure of finite extensions of Tararin groups}

\begin{notation} Through the rest of this section, when we write $G=T\rtimes\Z_n$, we mean implicitly that $T$ is a Tararin group with convex series $\{1\}=T_{m+1}\lhd\ldots \lhd T_0=T$, and that the generator of $\Z_n$ acts on $T/T_1$ by inversion. 
\end{notation}

In this subsection, we  analyze the algebraic structure of a group $G=T\rtimes \Z_n$. Our first observation is a lemma borrowed from \cite{rivas comm}. It will allow us to define the notion of  {\em base} and {\em inverting base}, our main computational tool in describing the structure of $G$.

\begin{lem}\label{lem useful} Let $G =T \rtimes\Z_n$.  Then, for $b \in T_{m}$ and $a\in T$, we have $aba^{-1}=b^\epsilon$ where $\epsilon=\pm 1$.
\end{lem}

\proof Let $b\in T_{m}$, $b\neq id$, and $a\in T$.  Then $a \in T_i$ for some $i=0,1,\ldots, m$.  We shall proceed by descending induction on $i$. The case $i=m$ is obvious since in this case $a$ and $b$ commute.  For the inductive step, assume that for any $w \in T_{i+1}$, we have $ wbw^{-1}=b^\epsilon$ for $\epsilon=\pm1$, and let $a \in T_i \setminus T_{i+1}$.  

From the structure of $G$, we know that there is $c\in G$ whose action by conjugation is multiplication by a negative  rational number on $T_i/T_{i+1}$: when $i >0$ the existence of $c$ follows from the definition of the Tararin group $T$, and when $i=0$ we take $c$ to be the generator of $\Z_n$. Thus there are positive integers $p,q$ such that $ca^{-p}c^{-1}=a^{q}w$, for some $w\in T_{i+1}$.

 Let $t \in \Q$ be such that $cbc^{-1}=b^t$, and suppose that $aba^{-1}=b^r$ for some $r \in \Q$. We will show $r = \pm 1$.  We have then that $a^nba^{-n}=b^{r^n}$ for all $n\in \Z$. Further we may calculate:
$$b^{r^{q}}=a^qba^{-q} = ca^{-p}c^{-1} w^{-1} b wca^pc^{-1}=ca^{-p}b^{\epsilon/t}a^pc^{-1}=cb^{\frac{\epsilon r^{-p} }{t }}c^{-1}=b^{\epsilon r^{-p}}.$$
Thus, $|r^{-p}| = |r^q|$, which implies that $r = \pm 1$.  $\hfill\square$

\vs
Applying Lemma \ref{lem useful} to the group $G/T_{i+1} \cong (T/T_{i+1}) \rtimes \Z_n$, one concludes

\begin{coro}  \label{cor useful}
For $i =0, \ldots, m-1$ the action of each $g\in G$ on $T_i / T_{i+1}$ is either trivial, or by inversion.  
\end{coro}

\begin{defi} Let $G=T\rtimes\Z_n$ where $T$ is a Tararin group with maximal convex subgroup $T_1$ and $\Z_n$ acts by inversion of $T/T_1$. We say that a subset $\{c, a_0,\ldots, a_m\}$ of $G$ is a {\bf\em base for $G$} if \begin{itemize}
\item $cT$ generates $G/T$,  $c$ has order $n$ in $G$, and the conjugation action of $c$ on each $T_i/T_{i+1}$ is by inversion.
\item For all $i=0,\ldots,m$, $a_i\in T_i\setminus T_{i+1}$ and $a_i$ acts by inversion of $T_{i+1}/T_{i+2}$.

\end{itemize}
If in addition 
\begin{itemize}
\item $a_i$ acts by inversion of $T_j/T_{j+1}$ for all $j>i$,
\end{itemize} we say that the base is an {\bf\em inverting base}.  If $\{c, a_0,\ldots, a_m\}$ is a base (resp. inverting base) of $G$, then we will refer to $\{a_0, \ldots, a_m\}$ as a base (resp. inverting base) of $T$.

\end{defi}

Since each of the quotients $T_i/T_{i+1}$ is rank one Abelian, $\mathrm{Aut}(T_i/T_{i+1})$ is also Abelian, and its torsion subgroup is isomorphic to $\mathbb{Z}_2$ (identifying  $T_i/T_{i+1}$ with a subgroup of $\mathbb{Q}$, the unique nontrivial finite order automorphism is multiplication by $-1$).  

Let $A_i$ denote the torsion subgroup of $\mathrm{Aut}(T_i/T_{i+1})$, and let $\phi_i : G \rightarrow \mathrm{Aut}(T_i/T_{i+1})$ be the conjugation action of $G$ on $T_i/T_{i+1}$ (recall that each $T_i$ is normal in $G$).  
By Corollary \ref{cor useful}, the image of $\phi_i$ lies in $A_i$.  Since this action will play a major role in our following work, we set some further notation.  

\begin{notation}
Let $\phi: G \rightarrow A_0 \times \ldots \times A_{m-1}$ be the homomorphism given by $$\phi(g) = (\phi_0(g), \ldots, \phi_m(g)).$$ Let 
 $\mathcal A = \ker(\phi)$.  Note that $\mathcal A$ is a subgroup of $T$ and admits a natural filtration $\mathcal A_i =\mathcal A\cap T_i$.  Moreover, $\mathcal A_i$ is the kernel of the restriction of $\phi $ to $T_i\rtimes \Z_n$.
\end{notation}

\begin{prop}The group $G=T\rtimes \Z_n$ admits an inverting base.
\end{prop}
\begin{proof}  By the definition of Tararin groups, and since $\mathbb{Z}_n$ acts by inversion on $T/T_1$,  there exists $a_i \in T_{i-1} \setminus T_{i}$ for $i >0$ and $a_0 \in G \setminus T$ that acts by multiplication by a negative rational (and thus by inversion) on $T_{i}/T_{i+1}$. Thus $\phi(a_i) =( 0, \ldots, 0, 1, *, \ldots, *) \in A_0 \times \ldots \times A_{m-1}$, where the stars indicate arbitrary entries, and the `1' appears in the $i$-th position.  Since the set of such elements form a generating set for $ A_0 \times \ldots \times A_{m-1}$, $\phi$ must be surjective.    Now let $c \in G$ be such that $\phi(c) = (1, \ldots, 1)$.  Then $c = tz$ for some $t \in T$, and $z$ a generator of $\Z_n$, and $c^n$ belongs to $T$, say $c^n\in T_i\setminus T_{i+1}$. If $c^n$ is not trivial, then, $cc^nc^{-1}=c^{-n}$ (mod $T_{i+1}$), but $c$ and $c^n$ commute, so $c^{-n}=c^n$ (mod $T_{i+1}$), contradicting the fact that $c^n\notin T_{i+1}$. Thus $c$ has orden $n$ and $G$ has an inverting base.  
\end{proof}

In general, an automorphism of $T$ fixing each successive quotient $T_i/T_{i+1}$ may not be the identity (this already happens for $T= \langle a,b\mid aba^{-1}=b^{-1}\rangle$). However, this is not the case if we demand that the automorphism fix the elements of an inverting base.

\begin{lem}
\label{lem identity} Let $G=T\rtimes\Z_n$, and $\{c,a_0,\ldots,a_m\}$ an inverting base. If $\phi : T \rightarrow T $ is an an automorphism of $T$ such that $\phi(a_i) = a_i$ for $i=0, \ldots, m$, then $\phi$ is the identity.
\end{lem} 
\begin{proof}
We induct on the rank of $T$.  The proposition is true if $T$ is rank one.  Now assume $T$ is of rank $m$, that $\phi : T \rightarrow T$ fixes the inverting base $\{a_0, \ldots, a_m \}$ and that the induced homomorphism $\bar{\phi} : T /T_m \rightarrow T/T_m$ is the identity since it fixes the inverting base $\{a_0T_m, \ldots, a_{m-1}T_m \}$.  

Suppose there exists $x \in T_i \setminus T_{i+1}$ satisfying $\phi(x) \neq x$, note $x \notin T_m$ since $\phi(a_m) = a_m$ implies $\phi$ is the identity on $T_m$.  Choose $x$ such that $i<m$ is maximal. By the induction assumption $\phi(x) = xw$ for some $1 \neq w \in T_m$.  Choose $p,q$ coprime integers such that $a_i^p x^q \in T_{i+1}$.  Note that $q$ cannot be even, for then $p$ would be odd and $a_i^p x^q$ would act nontrivially by conjugation on $T_{i+1}/T_{i+2}$, contradicting $a_i^p x^q \in T_{i+1}$.  

Consider the action of $x$ by conjugation on $T_m$.  If it is trivial, then $\phi(a_i^p x^q ) = a_i^p (xw)^q = a_i^p x^q w^q$, contradicting maximality of $i$ since $w \neq 1$.  If $x$ acts by inversion on $T_m$, then as $q$ is odd we compute $\phi(a_i^p x^q ) = a_i^p (xw)^q = a_i^p x^q w$, again contradicting maximality of $i$.  Thus $\phi$ must be the identity.
\end{proof}

This enables us to show our central result in this section.

\begin{prop}\label{prop central} If $\{c,a_0,\ldots,a_m\}$ is an inverting base for $G=T\rtimes \Z_n$, then
\begin{enumerate}
\item The element $c^2$ is central in $G$, and $a_j^2$ is central in $T_j$. 
\item The subgroup $\mathcal{A} = \ker(\phi)$ is Abelian.
\item If $x\in \mathcal A$ then $cxc^{-1}=x^{-1}$, while  $a_ixa_i^{-1}=x^{-1}$ for all $x\in \mathcal{A}_{i+1}=\mathcal A\cap T_{i+1}$.
\item The subgroup $\mathcal{A}_i$ is rank $m-i+1$ with $\{ a_i^2, \ldots, a_m^2\}$ as a maximal linearly independent subset (over $\mathbb{Z}$).
\end{enumerate}
\end{prop}

\begin{proof} The centrality of $c$ follows from Lemma \ref{lem finite order autos}. To show the other assertions, we induct on the rank of the Tararin group $T$.  For Tararin groups of rank one the proposition is obvious since $c$ acts by inversion of the rank one Abelian group $T$.  Assume it holds for groups of rank $m-1$ or less, and let $G = T\rtimes\Z_n$ where $T$ has rank $m$.  Then since $T_1$ is rank $m-1$ with inverting base $\{c,a_1, a_2,\ldots,a_m\}$, the inductive hypothesis says that (1), (2) and (3) hold (we will deduce (4) after having proved (1)--(3) by induction), namely $a_j^2$ is central in $T_j$ for $j>0$, that $\mathcal A_1=T_1\cap \mathcal A$ is Abelian, and that $a_ixa_i^{-1}=x^{-1}$ for $i>0$ and $x\in \mathcal {A}_{i+1}$.     

We first prove (1), where we need only show that $a_0^2$ is central. Observe that $\{ c, a_0,\ldots , a_{m-1} \}$ is an inverting base of $G /T_m$, hence, by the inductive hypothesis applied to the Tararin group $T/T_m$, $a_0^2$ is central in $T_0/T_m$. By Lemma \ref{lem identity} it suffices to show that $a_0^2$ commutes with all $a_i$, so suppose there is $a_i \neq id$ such that $a_0^2a_ia_0^{-2}= a_iw$, where $1 \neq w\in T_m$ and $i$ is as large as possible.  In particular, this means that $a_0^2$ commutes with all elements of $T_{i+1}$ by Lemma \ref{lem identity}.

Let $x\in T_{i+1}$ be such that $a_0a_ia_0^{-1}=a_i^{-1}x$. Then $a_iw=a_0^2a_ia_0^{-2}=a_0 a_i^{-1} x a_0^{-1}=x^{-1}a_i a_0 x a_0^{-1}$.  Since $x = a_i a_0 a_i a_0^{-1}$ and thus $x \in \mathcal A_{i+1}$,  the inductive hypothesis yields $a_ixa_i^{-1}=x^{-1}$.  Thus from $a_iw=x^{-1}a_i a_0 x a_0^{-1}$ it follows that $w=xa_0xa_0^{-1}=a_0xa_0^{-1}x$  (the last equality is because $\mathcal A_1$ is Abelian and contains both $w$ and $x$, so we may conjugate by $x$). But then $a_0$ acts trivially by conjugation on $T_m$, since $a_0wa_0^{-1}= a_0^2 xa_0^{-2} a_0xa_0^{-1}=xa_0xa_0^{-1}=w$ (here we use $a_0^2 x a_0^{-1} = x$ since $x \in T_{i+1}$). This contradicts the fact that $a_0$ acts by inversion of $T_m$, thus $a_0^2$ is central in $T_0$.

We now show (3), which by our induction assumptions reduces to showing $cxc^{-1}=x^{-1}$ for all $x\in \mathcal A$ and $a_0xa_0^{-1}=x^{-1}$ for all $x\in \mathcal A_1$.  We only show the former, as the latter can be shown in a similar way.  Choose $x \in \mathcal{A}$, and apply the induction assumptions to the group $T/T_m$ to arrive at $cxc^{-1}=x^{-1}w$ for some $w\in T_m$.  Then since $c^2$ is central we calculate $x=(cx^{-1}c^{-1})\,(cwc^{-1}) =w^{-1} x w^{-1}=xw^{-2}$, where the last equality is because $w\in T_m$ and $x \in \mathcal{A}$, and thus $x$ acts trivially on $T_m$. This contradicts the fact that $T_m$ is torsion free.

We next show (2).  Let $x,y\in \mathcal A$.  As before, consider the quotient $T/T_m$ and apply the induction hypothesis to arrive at $xyx^{-1}=yw$ for some $w\in T_m$. Since $x$ acts trivially on $T_m$ by conjugation, we may conjugate this equation by $x^{-1}$ to arrive at $x^{-1}yx = yw^{-1}$, or $x^{-1}y^{-1}x = wy^{-1}$ upon inverting.  On the other hand, starting with $xyx^{-1}=yw$ and conjugating by $c$ yields $x^{-1}y^{-1}x = y^{-1} w^{-1}$, by applying (3).  Thus $y^{-1} w^{-1} = wy^{-1}$, and as $y$ acts trivially by conjugation on $T_m$, we deduce $w^2=1$ and thus $w=1$ as $T$ is torsion free.  Thus $\mathcal{A}$ is Abelian.

Last we show (4).  To see that $\mathcal{A}_j$ is of rank $m-j+1$, let $a \in \mathcal{A}_j$ be given.  Suppose $a \in T_i \setminus T_{i+1}$ for some $i \geq j$.  Since $T_i/ T_{i+1}$ is rank one, we may choose $p_i$ and $q_i$ relatively prime, and $w_{i+1} \in T_{i+1}$ such that $a^{q_i} = a_i^{2p_i}w_{i+1}$.  But now $w_{i+1} \in \mathcal{A}_{i+1}$ and so there exists $q_{i+1}, p_{i+1}$ relatively prime and $w_{i+2} \in T_{i+2}$ such that $w_{i+1}^{q_{i+1}} = a_{i+1}^{2p_{i+1}}w_{i+2}$, and thus $a^{q_iq_{i+1}} = a_i^{2p_iq_{i+1}}a_{i+1}^{p_{i+1}}w_{i+2}$.  Proceeding by induction, we may write $a$ as a product of $\{a_i^2, \ldots, a_m^2\}$.  Linear indepence over $\mathbb{Z}$ is clear, since each $a_j^2$ is nonzero in exactly one of the quotients $T_j/T_{j+1}$ for $j =i, \ldots, m$.
\end{proof}

\vs

Proposition \ref{prop central} readily implies an algebraic decomposition  of $T$ which does not hold for any Tararin group but only for those who admits a finite order automorphism. Namely, that $T$ fits in the short exact sequence
$$0\to \mathcal A_0 \to T\to \Z_2\times\ldots\times\Z_2\to0,$$
where the $\Z_2$ factors are generated by $a_i$ ($i=0,\ldots,m-1$). Moreover, since $\mathcal A_0$ is a torsion free Abelian group of rank $m+1$, it may be identified with a subgroup of the vector space $\Q^{m+1}$. In this way any automorphisms admits a unique (up to identification) extension to a linear automorphism of $\Q^{m+1}$. For instance, if we identify $a_i^2\in \mathcal A_0$ with $e_{i+1}=(0\ldots,  0,1,0,\ldots,0)\in \Q^{m+1}$, then the conjugation action of $a_i$ on $\mathcal A_0$ extends to the involution 
$$ \sigma_i:(x_0,\ldots,x_m)\mapsto(x_0,\ldots,x_{i},-x_{i+1},\ldots,-x_m).$$




\subsection{Reduction to $T \rtimes \Z_2$ case}

To simplify the proof of Theorem \ref{teo tararin circ} in this subsection we reduce to the case where the Tararin subgroup has index 2.  We start with a simple observation. 

\begin{lem} Every finite order automorphism of $T$ has order two.
\label{lem finite order autos}
\end{lem}
\begin{proof} Let $\psi$ be a finite order automorphism of $T$.  Since $T_i \subset T$ are characteristic subgroups, $\psi$ induces an automorphism of $T_i/T_{i+1}$.  As observed in the previous section, the torsion part of $\mathrm{Aut}(T_i/T_{i+1})$ has order 2, so $\psi^2$ acts trivially on each $T_i/T_{i+1}$.  It follows from Remark \ref{rem orders on Tararins} that $\psi^2$ preserves each left orders on $T$.  

If $\psi^2$ is not the identity, then for any left order on $T$, there is some $g \in T$ such that $g < \psi^2(g)$.   Since $\psi^2$ preserves $<$, we have then $g<\psi^2(g)<\ldots < \psi^{2n}(g)$ for any $n$, contradicting the fact that $\psi$ has finite order.   
\end{proof}

\begin{prop} \label{prop finite to one}
Let $G = T\rtimes \Z_n$, where $T$ is a Tararin group, and $z$ is a generator of $\Z_n$ acting by inversion on $T/T_1$.  
Then $z^2$ is central in $G$, and there is a finite-to-one correspondence between circular orders on $G$ and circular orders on $G/\langle z^2 \rangle$
\end{prop}

\begin{proof} The fact that $z^2$ is central follows from the Lemma \ref{lem finite order autos}.
Now, following a strategy similar to that in Construction \ref{construction quotient}, for each circular order on $G$ we can produce a  circular order on $ \bar{G} = G/\langle z^2 \rangle$ as follows.  Given a circular order $c$ on $G$, let $z_0 \in \langle z \rangle$ be the unique generator of $\langle z \rangle$ such that $id, z_0, z_0^2, ... , z_0^{n-1}$ is in positive cyclic order.   For $id \neq \bar{g} \in \bar{G}$, say that its \emph{minimal representative under c} is the unique element $g \in \bar{g}$ such that $c(id, g, z_0^2) = 1$, and say that $\id$ is the minimal representative for $\id \in \bar{G}$.    Now define a circular order $\bar{c}$ on $\bar{G}$ by 
$\bar{c}(\bar{g}, \bar{h}, \bar{k}) = c(g_{min}, h_{min}, k_{min})$ where $g_{min}$ denotes the minimal representative of $g$.  It is easily verified that this gives a well-defined circular order.  

We now claim that the map $c \mapsto \bar{c}$ is finite-to-one, more precisely, the number of distinct circular orders on $G$ that induce a given order $\bar{c}$ by the construction above is given by $\phi(n) | \Hom(G, \Z_{n/2}) |$, where $\phi$ is Euler's totient function.    To see this, fix $c \in \CO(G)$.  Let $\bar{c}$ be obtained by the construction above, and let $z_0$ be the generator of $\Z_n$ appearing in the construction.   For each $g \in G$, let $r(g) \in \Z_{n/2}$ be the integer such that $g = g_{min} z_0^{2r(g)}$.  

Consider first the other circular orders $c' \in \CO(G)$ such that $\bar{c}' = \bar{c}$, and such that $e, z_0, z_0^2, ... , z_0^{n-1}$ are in positive cyclic order under $c'$.  Let $r'(g) \in \Z_{n/2}$ be the integer such that $g = g_{min'} z_0^{2r'(g)}$, where $g_{min'}$ is the minimal representative of $g$ under $c'$.  
Let $\psi: G \to  \Z_{n/2}$ be given by $g \mapsto r(g) - r'(g)$.  We claim that this is a homomorphism.   Indeed, similar to Construction \ref{construction lift}, we have $r(gh) = r(g) + r(h) + f_{g,h}$ where 
$$f_{g,h} = \left\{\begin{array}{cl} 0 & \text{if $g=\id$ or $h=\id$ or $\bar{c}( \id, g, gh)=1$}
\\ 1 &  \text{if $gh=id$ $(g\not=id)$ or $\bar{c}(\id,gh,g)=1$.  }  \end{array}
\right.$$ 
and $r'(gh) = r'(g) + r'(h) + f_{g,h}$ for the same function $f$ (which depended only on $\bar{c}$).  

Note that this homomorphism $\psi$ is uniquely determined by $c'$;  if $c''$ is a different other such order, then $r'' \neq r'$, so the homomorphisms constructed as above will differ.   Moreover, given any homomorphism $\psi: G \to \Z_{n/2}$, one may define a circular order $c'$ on $G$ by 
$$c'(g, h, k) = c(g z_0^{2\psi(g)}, h z_0^{2\psi(h)}, k z_0^{2\psi(h)}).$$
Then $\bar{c}' = \bar{c}$, and applying the construction above recovers $\psi$.  
This shows that there is a bijective correspondence between $\Hom(G, \Z_{n/2})$ and the set of orders on $G$ inducing $\bar{c}$ on $G/z^2$ with the property that $e, z_0, z_0^2, ... , z_0^{n-1}$ are in positive cyclic order.  

Finally, since there are $\phi(n)$ distinct circular orders on $\Z_n$ and hence $\phi(n)$ possibilities for choice of $z_0$, this proves the claim.  
\end{proof}

Thus, in order to prove Theorem \ref{teo tararin circ}, it suffices to consider groups of the form $G = T\rtimes \Z_2$.

\subsection{Actions on the circle}

In order to understand all circular orders on $G=T\rtimes \Z_2$, we analyze faithful actions of $G=T\rtimes \Z_2$ on the circle. The main result of this subsection is Corollary \ref{kernel cor} which, together with Corollary \ref{ker loc ind}, gives an algebraic description of the linear part of a circular order on $G$.

To pass from orders to actions we make use of the following folklore proposition, which we hinted at in the introduction.   A proof can be found in \cite{MR}. For its statement recall from Section \ref{sec preliminaries} the natural {\em orientation cocycle}  $\ord:S^1\times S^1\times S^1\to \{\pm 1,0\}$. 
\begin{prop} \label{prop dyn real}
Let $G$ be any countable group.  Then $G$ is circularly orderable if and only if $G$ acts faithfully on the circle by orientation-preserving homeomorphisms.  Moreover, for each circular order $c$ on $G$, there is a canonical (up to conjugacy) faithful \emph{dynamical realization} homomorphism $\rho_c: G \to \Homeo_+(S^1)$ with the property that there is a {\em base point} $p\in S^1$ such that
$$c(g_1,g_2,g_3)=\ord(\rho_c(g_1)(p),\rho_c(g_2)(p),\rho_c(g_3)(p)).$$   
\end{prop}

Thus, it makes sense to talk about the dynamics of $\rho_c$ given a circular order $c$.

\paragraph{Amenable groups and rotation number.} 
There is essentially one dynamical invariant of circle homeomorphisms; this is the \emph{rotation number} of Poincar\'e.  

\begin{defi}
Identify $S^1$ with $\R/\Z$ and let $g \in \Homeo(S^1)$.   The $\R/\Z$-valued \emph{rotation number} of $g$ is defined by 
$$\rot(g) :=  \lim \limits_{n \to \infty} \frac{\tilde{g}^n(x)}{n} \text{ mod } \Z$$  
where $\tilde{g}$ is any lift of $g$ to a homeomorphism of $\R$ and $x$ is any point in $\R$.  (In particular, this limit exists and is independent of $x$.)
\end{defi} 

In general, $\rot$ is not a homomorphism, however it is a homomorphism when restricted to actions of amenable groups.  Precisely, we have:

\begin{prop}\label{prop amenable rot}
If $G$ is an amenable group acting faithfully on the circle by homeomorphisms, then $\rot: G \to \R/\Z$ is a homomorphism.   The kernel of $\rot$ acts with a global fixed point.  
\end{prop}
\noindent This follows from the existence of a $G$-invariant Borel probability measure on $S^1$. See \cite[Prop 6.17]{Ghys} for a proof.

Recall that a group is \emph{locally indicable} if every finitely generated subgroup admits a surjective homomorphism to $\Z$. An important theorem of Witte-Morris \cite{morris} states that any amenable group that acts faithfully by orientation preserving homeomorphisms on the line is locally indicable.  When applied in conjunction with Proposition \ref{prop amenable rot}, this gives:

\begin{coro}  \label{ker loc ind}
Let $G$ be an amenable group with a circular order, and let $G$ act on $S^1$ via the dynamical realization.  Then $\ker(\rot)$ is the linear part of the order, and $\ker(\rot)$ is a locally indicable group.  
\end{coro}

\begin{proof} 
Let $c$ be a circular order on $G$, and let $G$ act on $S^1$ by its dynamical realization with basepoint $p$.  
Let $K = \ker(\rot)$, and let $I$ be the open component of $S^1\setminus\{x\mid K(x)=x\} $ that contains $p$. Since $K$ is normal in $G$, for every $g\in G$ we have that $g(I)\cap I$ is either $I$ or the empty set.  It now follows easily from the definitions that $K$ is a convex subgroup of $G$.  Maximality follows since the action of $G$ is semi-conjugate to a faithful action of $G/K$ by rotations (as in the proof of Proposition \ref{prop amenable rot} in \cite{Ghys}), and so there are no convex subgroups containing $K$.   Local indicability follows from Witte-Morris' theorem, since $K$ acts faithfully with a global fixed point.  
\end{proof}

This is particularly relevant to our situation, because solvable groups are always amenable -- in particular, this holds for our (solvable) group $G = T \rtimes \Z_n$.

\paragraph{Tararin groups and convex subgroups.} 
The key observation in the proof that a Tararin group $T$ has only finitely many linear orders is that $T_1$ is always a \emph{convex subgroup}.  One might expect, by analogy, that in a circular order on $G = T\rtimes \Z_2$, the group $T$ would always be convex---if this were true, the proof of  Theorem \ref{teo tararin circ} would be much simpler.   However, the following example shows that $T$ is not necessarily a convex subgroup.  

\begin{exam}\label{ex weird tararin} Let $G=\langle a,b,c\mid c^2, aba^{-1}=b^{-1}, cac^{-1}=a^{-1}b\rangle \cong K_2\rtimes_{-1} \Z_2$, where $K_2=\langle a,b\mid aba^{-1}=b^{-1}\rangle$ is the Klein bottle group (in particular, it is a Tararin group with series $\{1\} \lhd \langle \langle b \rangle \rangle \lhd K_2$). We exhibit a circular order on $G$ whose linear part is not $\langle a,b\rangle$.

To do that, note that the map from $G$ to $\{0,1\}$ sending $a,c \mapsto 1$ and $b\to 0$ extends to a homomorphism form $G$ to $\Z_2$. Let $H$ be the kernel, then $H$ is generated by $a^2, b, ca$. We claim that $H$ is left orderable (so a circular order on $G$ with the desired properties can be constructed using the converse of Proposition \ref{prop extensions}(1)).

Observe that $caca=b^{-1}$, so $H$ is in fact generated by $ca$ and $a^2$. Further, the subgroup $H_1$ generated by $a^2$ is invariant under conjugation by $ca$, as we compute $ca a^2 (ca)^{-1}=a^{-2}$. One can check that $H$ is therefore a Tararin group isomorphic to $K_2$, with series $\{id\}\lhd H_1\lhd H$. 
\end{exam}

Despite this problem, our strategy for the proof of Theorem \ref{teo tararin circ} is to show that the linear part (i.e. maximal convex subgroup) of any circular order on a group of the form $T\rtimes \Z_2$ is indeed a Tararin group, even though it may not be the subgroup $T$.   We begin by studying the constraints on rotation numbers for actions of $G$ on $S^1$.  

\begin{lem} \label{lem inversion} Suppose that the group $G=T\rtimes \Z_2$ acts faithfully on $S^1$ by order preserving homeomorphisms. \begin{enumerate} 
\item If $a\in T$ satisfies that $cac^{-1}=a^{-1}$ for some $c\in G$ of order two, then $rot(a)=0$.
\item In general $rot(G)\subseteq \{0,1/2\}$.
\end{enumerate}

\end{lem}
\begin{proof} Let $c$ be an order two element of $G$. Suppose that $cac^{-1}=a^{-1}$ holds for some $a\in T$. This implies that $\rot(a)$ is either $0$ or $1/2$. Now,  since $G$ is solvable (hence amenable), Corollary \ref{ker loc ind} implies that $\rot:\langle a,c\rangle \to \R/\Z$ is a homomorphism with locally indicable kernel, in particular $c \notin \ker(\rot)$. Thus, if $\rot(a) = 1/2$ then $\rot(ca)=0$ since $\rot(c)=1/2$. But $(ca)^2=c^2=id$ and, since the $\ker(\rot)$ is torsion free, this gives $ca=\id$. But then $a=c$, contradicting the fact that $a$ has infinite order. Thus $\rot(a) =0$ and (1) holds.

To show (2) note that by Proposition \ref{prop central} we have $cxc^{-1}=x^{-1}$ for all $x\in \mathcal{A}$. Thus by the previous paragraph, $\rot(\mathcal{A})=0$. In particular, $\rot:G\to S^1$ factors though $G/\mathcal{A}\simeq \Z_2\times \ldots \times \Z_2$. It follows that $\rot(G)\subseteq \{0,1/2\}$.  
\end{proof}

We now analyze the kernel of the rotation number homomorphism given a faithful action of $G = T \rtimes \Z_2$ on $S^1$.   Since $\rot(G) \subset \{0, 1/2\}$ it suffices to understand homomorphisms $G \to \Z_2$ with locally indicable kernel (see Corollary \ref{ker loc ind}).

\begin{lem} \label{lem 0 and 1} Let $\psi:G\to\Z_2$ be a homomorphism with locally indicable kernel. Then $\psi(\mathcal A)=0$. Suppose further that $\{c,a_0,\ldots,a_m\}$ is a base for $G$  satisfying that, for all $k<i_0$, $a_k a_{i_0+1}a_k^{-1}=a_{i_0+1}^{-1} \; (mod \; T_{i_0+2})$ (an inverting base for instance). Then $\psi(a_{i_0})=1$ implies that  $\psi(a_i)=1$ for $0\leq i \leq i_0$.
\end{lem}

\begin{proof} The fact that $\psi(\mathcal A)=0$ follows as in the proof of Lemma \ref{lem inversion} above, using the fact that $cxc=x^{-1}$ for all $x\in \mathcal A$.

Let $i_0$ be the largest index such that $\psi(a_{i_0})=1$, and suppose $\psi(a_k) = 0$ for some $k<i_0$.  Note that $\psi(c) = 1$ since $c$ is torsion and $\ker(\psi)$ is torsion-free, therefore $\psi(ca_{i_0})=0$.  Consider the subgroup $H$ of $\ker(\psi)$ generated by $a_k$ and $ca_{i_0}$, and suppose that there is a surjective homomorphism $\phi: H \rightarrow \mathbb{Z}$.

By Proposition \ref{prop central}(3), since $(ca_{i_0})^2 \in \mathcal{A}_{k+1}$, we have $a_k (ca_{i_0})^2 a_k^{-1} = (ca_{i_0})^{-2}$, forcing $\phi(ca_{i_0}) = 0$.  On the other hand, $(ca_{i_0})a_k(ca_{i_0})^{-1} = a_k^{-1} w$ for some $w \in T_{k+1} \cap H$.  But then $w \in \mathcal{A}_{k+1}$, so $a_k w a_k^{-1} = w^{-1}$ and $\phi(w)=0$.  Applying $\phi$ to both sides of $(ca_{i_0})a_k(ca_{i_0})^{-1} = a_k^{-1} w$ then implies that $\phi(a_k) =0$ as well, a contradiction.

\end{proof}

As a consequence of this we have the following result (compare with the remarks after Proposition \ref{prop central}).  

\begin{coro} \label{kernel cor}
Suppose $T$ is a rank $m+1$ Tararin group, and $G = T \rtimes \Z_2$ acts faithfully on the circle by order preserving homeomorphisms. Let $H = \ker(\rot)$.  
Then $H$ sits in an exact sequence
$$0\to \mathcal A \to H\to \underbrace{\Z_2\times\ldots\times\Z_2}_{m \text{ copies}} \to 0,$$
where $\mathcal A$ has rank $m+1$, and the generator of the $i^{th}$ $\Z_2$ factor ($i=0,\ldots,m-1$) acts by conjugacy on $\mathcal A$ as
\begin{equation}\label{eq involutions} \sigma_i:(x_0,\ldots,x_m)\mapsto (x_0,\ldots,x_{i},-x_{i+1},\ldots, -x_m ).\end{equation}
\end{coro} 

\begin{proof}

Lemma \ref{lem 0 and 1} implies that $\mathcal A \subset H$, so we have $H \to H/\mathcal A \subseteq G/\mathcal A = \Z_2\times\ldots\times\Z_2$.   We show this map is surjective and that \eqref{eq involutions} holds.  Fix $\{c,a_0,\ldots,a_m\}$ an inverting base for $G$. If $\rot(a_i)=0$ for all $i=0,\ldots,m$, then $H=T$, and we are done by the remarks after Proposition \ref{prop central}. 
If $H\neq T$, then by Lemma \ref{lem 0 and 1} there exists $i_0$ with $0 \leq i_0 < m$ such that
$$
\rot(a_i)=
\begin{cases}
0 \text{ if } i> i_0 \\
1 \text{ if } i\leq i_0.
\end{cases}
$$

Then, if we identify $a_i^2\in \mathcal A$ with $e_{i+1}\in \Q^{m+1}$ ($0$'s everywhere except in position $i+1$) then under conjugacy, for $j=i_{0}+1,\ldots,m-1$, 
$$ca_{i_0} a_j: (x_0,\ldots,x_m)\mapsto (-x_0,\ldots, -x_{i_0},x_{i_0+1},\ldots,x_j,-x_{j+1},\ldots,-x_m) $$
and for $j=0,\ldots,i_0-1$, 
$$a_{i_0}a_j: (x_0,\ldots,x_m)\mapsto (x_0,\ldots, x_j,-x_{j+1},\ldots,-x_{i_0},x_{i_0+1},\ldots,x_m).$$
After a reindexing  the basis of $\Q^{m+1}$ this is precisely equation (\ref{eq involutions}). 
\end{proof}

\subsection{Finishing the proof Theorem \ref{teo tararin circ}}
\label{sec finishing}

Let $G = T \rtimes \Z_2$ as before.    We summarize the results of the previous subsection.  

For any circular order on $G$, the dynamical realization gives a function $\rot:G\to S^1$, which is a homomorphism since $G$ is amenable.  In Proposition \ref{prop central}, we showed that $\rot$ factors through the finite group $G/\mathcal A$, so there are only finitely many possible rotation number homomorphisms on $G$.   Recall also that $\ker(\rot)$ is the linear part of the order on $G$.

Thus, to finish the proof of Theorem \ref{teo tararin circ} it is enough to show that 
the kernel of $\rot$ always admits only finitely many linear orders, i.e. is a Tararin group.  Corollary \ref{kernel cor} reduces this to proving the following proposition.    Its proof completes the proof of Theorem \ref{teo tararin circ}.

\begin{prop}\label{prop tararin char} Suppose $H$ is a locally indicable group satisfying 
$$0\to A\to H\to \Z_2\times \ldots \times \Z_2\to 0,$$
where $A\subseteq \Q^{m+1}$ is a rank $m+1$, torsion free Abelian group and there are $m$ copies of $\Z_2$. Suppose also that the $i^{th}$ $\Z_2$ factor ($i=0,\ldots,m-1$) acts on $\Q^{m+1}$ (by the induced conjugacy action on $A$) as
\begin{equation}\label{eq involutions2} \sigma_i:(x_0,\ldots,x_m)\mapsto (x_0,\ldots,x_{i},-x_{i+1},\ldots, -x_m ).\end{equation}
Then $H$ is a Tararin group.
\end{prop}

\begin{rem} In Proposition \ref{prop tararin char} and in Lemma \ref{lem 0 and 1}, local indicability is crucial. For instance, the Promislow group $P:=\langle a,b\mid ab^2a^{-1}=b^{-2}, \; ba^2b^{-1}=a^{-2}\rangle$ is a torsion free group sitting in the short exact sequence 
$$0\to \Z^3\to P\to \Z_2\times \Z_2\to 0,$$ where $\Z^3$ is generated by $a^2$, $b^2$ and $(ab)^2$;  the first $\Z_2$ factor acts on $\Z^3$ by conjugation by $a$ as $(x_0,x_1,x_2)\mapsto (x_0,-x_1,-x_2)$, and the second $\Z_2$ factor is conjugation by $b$ via $(x_0,x_1,x_2)\mapsto (-x_0,x_1,-x_2)$.    The group $P$ is solvable, but not locally indicable, so not a Tararin group.  See 
 \cite[\S1.4]{GOD} for more details.  

\end{rem}

  We recall the following notion that will be used in the proof of the proposition. If $G_1$ is a subgroup of $G_2$, then the {\bf\em isolator} of $G_1$ on $G_2$ is 
$Isol_{G_2}(G_1)=\{g\in G_2\mid g^n\in G_1 \text{ for some }n\in \Z\setminus\{0\}\}.$ If $g_1,\ldots,g_n$ belong to $G$, we write $Isol_{G}(g_1,\ldots,g_m)$ to mean the isolator on $G$ of the group generated by $g_1,\ldots,g_n$.

\begin{proof}[Proof of Proposition \ref{prop tararin char}]
Let $H$ be group satisfying the hypothesis of Proposition \ref{prop tararin char}.  Let $e_0,\ldots,e_m$ be a base of $\Q^{m+1}$ implicit in (\ref{eq involutions2}). By eventually taking powers, we may assume also that the $e_i's$ also belong to $A$. Let $Fix(\sigma_i)=\{a\in A\mid \sigma_i a\sigma_i^{-1}=a\}$. Clearly, $\sigma_i^2\in Fix(\sigma_i)=Isol_A( e_0,\ldots,e_i)$. 

To prove that $H$ is a Tararin group, we are going to show that:
 \begin{enumerate}
\item The group $H$ admits a unique (up to multiplicative constant) non trivial homomorphism $L:H\to\R$, and $L(H)$ has rank one.
\item The subgroup $H_1=ker(L)$ is a locally indicable group satisfying the hypothesis of Proposition \ref{prop tararin char}. In particular, there is a unique non trivial homomorphism $L_1:H_1\to \R$, and its image has rank one.
\end{enumerate}

Given this, it follows by induction that $H$ admits a filtration $H\rhd H_1\rhd ...\rhd H_\ell =\{id\}$ with each quotient $H_i/H_{i+1}$ a torsion free, rank one Abelian group.  Further, in each $H_i/H_{i+1}$ there must be an element whose action by conjugation on $H_{i+1}/H_{i+2}$ is non-trivial, as otherwise $H_{i}/H_{i+2}$, being a rank two Abelian group, would not admit a unique torsion free Abelian quotient.  Moreover, since $H$ has a finite index Abelian subgroup, the action by conjugation of that element must  have finite order, so it must be by multiplication by $-1$. We conclude that $H$ is a Tararin group.  

It remains to show 1 and 2.  We start with a Lemma on the structure of virtually Abelian locally indicable groups.  

\begin{lem}\label{lem finite rank} Suppose $H$ is group having a finite index Abelian group $A$ of finite rank. If $H$ is locally indicable, then $H$ admits a torsion-free Abelian quotient.

\end{lem}

\proof
Since $H$ is countable (it is a finite rank solvable group) and locally indicable, it is left orderable, and so it acts  faithfully on the line by homeomorphisms via the dynamical realization of any of its left orders \cite{GOD,Ghys,MR} (this is the analog of Proposition \ref{prop dyn real} but for left orders). Certainly, we may assume that this action is without global fixed points.  Since the finite index Abelian subgroup $ A$ has finite rank, $A$ preserves a Radon measure $\mu$ on the line, see \cite{plante}. 
Since $H/ A$ is finite, the barycenter $\nu$ of  $H_*\mu$ is an $H$-invariant Radon measure on the line. This gives a translation number homomorphism to $\tau:H\to \R$ defined by $\tau(g)=\nu([0,g(0))$. We claim that this homomorphism is not trivial. Indeed, if $\tau(g)=0$ for all $g\in H$ then each element of $H$ has a fixed point and, as $A$ is finite rank Abelian, $A$ has a global fixed point $p\in \R$. Since $A$ is finite index, and the $H$ action preserves the ordering on the line, $p$ is a global fix point of $H$, contradicting the fact that dynamical realization actions has no global fixed points. Thus $\tau:H\to \R$ is a non trivial hmomorphism.  \hfill $\square$

\vs

Now we apply this to our group $H$, keeping the notation from above.  
Since $\sigma_0^2\in Isol_A(e_0)$, by eventually renaming $e_0$, we can assume $e_0=\sigma_0^2$. Let $L:H\to \R$ be a non trivial homomorphism provided by Lemma \ref{lem finite rank}. Since $\sigma_0$ inverts $e_1,\ldots,e_m$, $L$ factors through $H/Isol_H(e_1,\ldots,e_m)$ which has rank equal to 1. In particular $L(H)$ has rank one, and $L$ is unique up to multiplicative constant.  This proves (1) above.  Note also that $L(\sigma_0)\neq 0$, so we may assume that $L(\sigma_0)=1$.

We now show (2). Let $H_1$ be the kernel of $L$. Observe that $A\cap H_1=Isol_A(e_1,\ldots,e_m)$. Let $L(\sigma_i)=p_i/q_i$ be a reduced fraction. Then $L(\sigma_i^{q_i})=p_i$ so that $L(\sigma_i^{q_i} \sigma_0^{-p_i})=0$.

We first show that each $q_i$ is odd.  Suppose for contradiction that some $q_i$ is even.  Then  $\sigma_i^{q_i} \in A$, and we may write it as $\sigma_i^{q_i} = y_0y$ where $y_0\in Isol_A(e_0)$ and $y\in Isol_A(e_1,\ldots,e_i)$.  Then 
$$(\sigma_i^{q_i}\sigma_0^{-p_i})^2= y_0y\sigma_0^{-2p_i} \sigma_0^{p_i} y_0 y \sigma_0^{-p_i}= y_0^2\sigma_0^{-2p_i} \in Isol_A(e_0)$$
since $p_i$ is necessarily odd, so $\sigma_0^{p_i} y_0 \sigma_0^{-p_i} = y_0^{-1}$.   However, since $\sigma_0^2 \in Isol_A(e_0)$ as well, there exist $s, t$ such that $\sigma_0^{2s} = (\sigma_i^{q_i}\sigma_0^{-p_i})^{2t}$, and thus $L(\sigma_0^{2s}) = L((\sigma_i^{q_i}\sigma_0^{-p_i})^{2t}) =0$, contradicting the fact that the image of $L$ is torsion free.
 We conclude that no $q_i$ is even.  

Let $\sigma'_i:=\sigma_i^{q_i}\sigma_0^{-p_i}$, then $\sigma_i' \in H_1=ker(L)$.  If $p_i$ is even, the action of $\sigma'_i$ agrees with the action of $\sigma_i$ on $A$.  If $p_i$ is odd, then instead $\sigma_i'\in H_1$ inverts $e_1,\ldots, e_i$

In summary we have 
\begin{equation}\label{eq key} \begin{array}{c} p_i \text{ even, then $\sigma_i'\in H_1$ inverts $e_{i+1},\ldots,e_m$}\\
p_i \text{ odd, then $\sigma_i'\in H_1$ inverts $e_1,\ldots, e_i$}\end{array}\end{equation}

Suppose as a first case that $p_1$ is even.  We claim then that all $p_i$ are even. Indeed, suppose that some $p_i$ is odd where $i\geq 2$.  Then $\sigma'_1$ inverts $e_2,\ldots, e_m$ and $\sigma_i'$ inverts $e_1$ (among other things). In particular, any element of $A\cap H_1$ is inverted by some element of $H_1$. Thus, since $A\cap H_1$ has finite index in $H_1$, we conclude that $H_1$ is not indicable, contradicting Lemma \ref{lem finite rank}. (Recall that, by assumption, the group $H_1$ is locally indicable.)  Thus, if $p_1$ is even, then all $p_i$ are even.   In this case, we have a short exact sequence 
$$0\to A \cap H_1 \to H_1 \to \underbrace{\Z_2 \times ... \times \Z_2}_{m-1} \to 0$$
where the homomorphism to the $i^{th}$ $\Z_2$ factor is given by the conjugation action of $\sigma'_i$ on $A \cap H_1$.   The group $A \cap H_1$ inherits a basis from $A$, and equation \eqref{eq key} implies that the conjugation action of $\sigma'_i$, in this basis, is exactly that given in Proposition \ref{prop tararin char}.  Thus, $H_1$ satisfies the hypotheses of the proposition.

To finish we analyze the  case where $p_1$ is odd, so $\sigma_1'$ only inverts $e_1$.  If all $p_i$ are odd, we define a surjection $H_1 \to \underbrace{\Z_2 \times ... \times \Z_2}_{m-1}$ where the homomorphism to the $i^{th}$ $\Z_2$ factor $(i = 0, 1, ... m-2)$ is given by the conjugation action of $\sigma'_{m-1-i}$ on $A \cap H_1=Isol_A(e_1,\ldots,e_m)$.   In coordinates on $A \cap H_1$ induced from the ordered basis $(e_m, e_{m-1}, ...,  e_1)$ this action agrees with that in the statement of Propostion \ref{prop tararin char}. 

Otherwise, let $j$ be the minimum integer such that $p_j$ is even.  Then, just as above, $p_i$ must be even for all $i \geq j$ -- if not, then Equation \eqref{eq key} implies any element of $A \cap H_1$ is inverted by some element $\sigma_k'$ of $H_1$, contradicting indicability.  
In this case, we define $H_1 \to \underbrace{\Z_2 \times ... \times \Z_2}_{m-1}$ where the homomorphism to the $i^{th}$ $\Z_2$ factor $(i = 0, 1, ... m-2)$ is given by conjugation by $\sigma'_{j-1} \sigma'_{j+i}$ for $i = 0, 1, m-j$,  by $\sigma'_{j-1}$  for $i = m-j+1$, and for $i = m-j+1, ..., m-2$ by $\sigma'_{j-1} \sigma'_{i- (m-j+1)}$.   In coordinates on $A \cap H_1$ induced from the ordered basis $(e_j, e_{j+1}, e_{m-1}, e_1, e_2... e_{j-1})$ this action agrees with that in the statement of Proposition \ref{prop tararin char}.   This shows (2), and completes the proof. 
\end{proof}

\subsection{Computing $|\CO(G)|$}  \label{number sec}

We have shown that $G=T\rtimes \Z_n$ admits only finitely many circular orders. However, unlike in the case of left orders on $T$ where $|\LO(T)|=2^{rank(T)}$ (see Remark \ref{rem orders on Tararins}), computing the cardinality of $\CO(G)$ seems to be a difficult task. Proposition \ref{prop extensions} gives a lower bound, and the same proposition together with our work gives  some {\em obvious} upper bound:
$$\phi(n)\, 2^{rank(T)} \leq |\CO(G)|\leq \phi(n)\, 2^{rank(T)}  \; k_{morph}\; k_{lift}.  $$
Here $\phi$ is Euler's totient function, $k_{lift}$ is the number of homomorphisms from $G$ to $\Z_{n/2} $ (see the proof of Proposition \ref{prop finite to one}), and $k_{morph}$ is the number of homomorphisms from $G$ to $\Z_2$---all of which factor through the finite group $G/\mathcal A$---(see the beginning of Section \ref{sec finishing}).

The upper bound is never realized since, as we saw in Corollary \ref{ker loc ind}, the linear part of a circular order must be locally indicable, and this does not hold for every homomorphism from $G$ to $\Z_2$. In fact, Proposition \ref{prop extensions} together with the discussion above provide the following expression for the cardinality of $\CO(G)$ 
$$|\CO(G)|= \phi(n)  2^{rank(T)} \; k_{rot}\; k_{lift},$$
where $k_{rot}$ is the cardinality of the set of homomorphisms from $G$ to $\Z_2$ having a locally indicable kernel (equivalently, $k_{rot}$ is the cardinality of the set of possible rotation number homomophisms $rot:G\to S^1$). This ``formula" is somewhat unsatisfactory because giving a precise expression for $k_{rot}$ in terms of the algebraic structure of $G$ seems a challenging task.

On the other hand, the lower bound for $|\CO(G)|$ is sharp.  For groups $G =T\rtimes \Z_n$ where the $\Z_n$ factor acts by inversion of $T$, Lemma \ref{lem inversion} implies that $T$ is always convex (equivalently, $k_{rot}=1$) and thus  $|\CO(G)|=\phi(n) |\LO(T)|$ (an explicit example could be $G=\Z\rtimes \Z_2$). However, this lower bound is not always achieved; Example \ref{ex weird tararin} gives a group where $\phi(n)\, 2^{rank(T)} < |\CO(G)|$ holds.

\section{$\CO(G)$ is finite or uncountable}  \label{cardinality sec}

In this section we prove that any group having only finitely many circular orderings admits a semi-direct product decomposition $T\rtimes\Z_n$ as in Section \ref{sec Tararins}.  This will be appear as part of the proof of the following theorem.
\begin{thm} \label{teo cardinality} Let $G$ be any group. Then
 $\CO(G)$ is either finite or uncountable. 
\end{thm}

\noindent{\em Proof.} We follow the Linnell's argument from \cite{linnell}. Let $G$ be a circularly orderable group and consider the action of $G$ on $\CO(G)$ by conjugation.  Let $M$ be a minimal set of this action.  If $M$ is infinite, then it is a Cantor set. So we assume it is finite, in particular there exists some $c \in \CO(G)$ with finite conjugation orbit.  Let $K \subset G$ be the stabilizer of $c$ under the $G$-action on $\CO(G)$ by conjugation. Then $K$ is a finite index subgroup of $G$, and the restriction of $c$ to $K$ is bi-invariant.  Thus, Corollaries \ref{coro holder} and \ref{finite index} imply that, if we denote by $H$ the linear part of $c$, then $G/H$ is order isomorphic to a subgroup of $S^1$. In particular, $G/H$ is Abelian.  

Observe that the natural embedding of $\LO(H) \times \CO(G/H)$ into $\CO(G)$ has $c$ in its image. Therefore, if $G/H$ is infinite, then Theorem \ref{teo abelian} implies that $G/H$ admits uncountably many circular orderings. So the same holds for $\CO(G)$. Further, if $H$ is not a Tararin group, then $\LO(H)$ is uncountable \cite{linnell}, hence the same holds for $\CO(G)$. 

Thus we restrict our attention to the case where $G/H$ is a finite (hence cyclic) group and $H$ is a Tararin group. We denote by $H_1$ the maximal proper convex subgroup (in any left order) of $H$. The subgroup $H_1$ is normal in $G$ since it is a characteristic subgroup of $H$.  In particular we have that $c$ is in the image of the natural embedding 
$$ \LO(H_1)\times \CO(G/H_1)\to \CO(G).$$ 

We analyze two separate cases:
  \begin{enumerate}
\item The conjugation action of $G/H$ on $H/H_1$ is trivial, that is, $G/H_1$ is Abelian.

In this case Theorem \ref{teo abelian} implies that $\CO(G/H_1)$ is uncountable so the same holds for $\CO(G)$.

\item The conjugation action of $G/H$ on $H/H_1$ is not trivial. We claim that in this $G\simeq H\rtimes G/H$, and the action of $G/H$ on $H/H_1$ is by multiplication by $-1$.

Let $g\in G\setminus H$ such that $g$ acts non trivially on $H/H_1$. Since $G/H$ is finite and $H/H_1$ is a rank one abelian group, the conjugation action of $g$ on $H/H_1$ is by multiplication by $-1$, as this is the only automorphism of finite order. 
By eventually picking a different $g'$ in $gH$, we may assume that $g$ also acts non-trivially on each $H_i/H_{i+1}$ for $i=0,1,\ldots$, where $H=H_0\rhd H_1\rhd\ldots \rhd H_k=\{id\}$ is the rational series of the Tararin group $H$.  
Then, if $n$ is the cardinality of $G/H$, $g^n\in H$, but $g$ commutes with $g^n$, so $g^n=id$. Therefore 
$$G= H \rtimes \Z_n $$
where $n$ is even and $\Z_n$ acts by multiplication by $-1$ on $H/H_1$.  In particular, by Theorem \ref{teo tararin circ}, $\CO(G)$ is finite.
$\hfill \square$
 \end{enumerate}

\begin{small}

\vsp

\vsp\vsp\vsp

\noindent Adam Clay\\
Dept. Mathematics, University of Manitoba\\
420 Machray Hall, Winnipeg, MB, Canada R3T 2N2\\
Adam.Clay@umanitoba.ca

\vsp
\vsp

\noindent Kathryn Mann \\
Dept. Mathematics, UC Berkeley\\
970 Evans hall, Berkeley, CA, 94608, USA\\
kpmann@math.berkeley.edu

\vsp
\vsp

\noindent Crist\'obal Rivas \\
Dpto. de Matem\'atica y C.C.,Universidad de Santiago de Chile\\
Alameda 3363, Estaci\'on Central, Santiago, Chile \\
cristobal.rivas@usach.cl
\end{small}

\end {document}